\documentclass[10pt]{amsart}
\usepackage{graphicx}
\usepackage{amscd}
\usepackage{amsmath}
\usepackage{amsthm}
\usepackage{amsfonts}
\usepackage{amssymb}
\usepackage{mathrsfs}
\usepackage{enumerate}
\usepackage[pagebackref,hypertexnames=false, colorlinks, citecolor=blue, linkcolor=red, pdfstartview=FitB]{hyperref}
\usepackage{amsrefs}

\usepackage[latin1]{inputenc}

\newcommand{\D}{\operatorname{\mathbb{D}}}
\newcommand{\Dir}{\operatorname{\mathcal{D}}}

\newcommand{\R}{\operatorname{\mathbb{R}}}

\newcommand{\C}{\operatorname{\mathbb{C}}}
\newcommand{\T}{\operatorname{\mathbb{T}}}

\newcommand{\A}{\operatorname{\mathcal{A}}}
\newcommand{\E}{\operatorname{\mathcal{E}}}

\newcommand{\hil}{\operatorname{\mathcal{H}}}
\newcommand{\kil}{\operatorname{\mathcal{K}}}

\newcommand{\ol}{\overline }
\let\phi\varphi
 \DeclareMathOperator{\Ext}{Ext}
\DeclareMathOperator{\ran}{ran}
\newtheorem{lemma}{Lemma}[section]
\newtheorem{theorem}[lemma]{Theorem}
\newtheorem{proposition}[lemma]{Proposition}
\newtheorem{corollary}[lemma]{Corollary}
\theoremstyle{definition}

\begin{document}
\author{Rapha\"el Clou\^atre}
\address{Department of Pure Mathematics, University of Waterloo, 200 University Avenue West,
Waterloo, Ontario, Canada N2L 3G1} \email{rclouatre@uwaterloo.ca}
\title[Spectral and homological properties of Hilbert modules]{Spectral and homological properties of Hilbert modules over the disc algebra}
\subjclass[2010]{46H25, 47A10}
\keywords{Hilbert module over the disc algebra, projective Hilbert module, spectrum, extension group, derivations}
\begin{abstract}
We study general Hilbert modules over the disc algebra and exhibit necessary spectral conditions for the vanishing of certain associated extension groups. In particular, this sheds some light on the problem of identifying the projective Hilbert modules. Part of our work also addresses the classical derivation problem.
\end{abstract}
\maketitle

\section{Introduction}
This paper is concerned with polynomially bounded operators and some of their spectral properties. Recall that a bounded linear operator $T$ acting on some Hilbert space $\hil$ is said to be
polynomially bounded if there exists a constant $C>0$ such that for
every polynomial $p$, we have
$$
\|p(T)\|\leq C\|p\|_{\infty}
$$
where
$$
\|p\|_{\infty}=\sup_{|z|<1}|p(z)|.
$$
This inequality allows one to extend continuously the polynomial
functional calculus $p\mapsto p(T)$ to all functions $f$ in
the disc algebra $A(\D)$, which consists of the holomorphic
functions on $\D$ that are continuous on $\ol{\D}$ (throughout the
paper $\D$ denotes the open unit disc and $\T$ denotes the unit
circle). The point of view we adopt is that of Douglas and Paulsen (see \cite{douglas1989}) where these operators are studied as modules over the disc algebra: the map
$$
A(\D)\times \hil\to \hil
$$
$$
(f, h)\mapsto f(T)h
$$
gives rise to a module structure on $\hil$, and we say
that $(\hil, T)$ is a \emph{Hilbert module}. We only deal with modules over $A(\D)$ in this paper,
so no confusion may arise regarding the underlying function algebra
and we usually do not mention it explicitely. Moreover, when the
underlying Hilbert space is understood, we slightly abuse
terminology and say that $T$ is a Hilbert module.
Using these notions, the authors of \cite{douglas1989} reformulated several interesting operator theoretic questions in the language of module theory, and in doing so suggested the use of cohomological methods. Accordingly, we phrase most of our results using extension groups of Hilbert modules, and thus we briefly review the definition of these groups.

Given two Hilbert modules $(\hil_1,T_1)$ and $(\hil_2,T_2)$, the extension group $$\Ext_{A(\D)}^1(T_2,T_1)$$ consists of equivalence classes of exact sequences
$$
0 \to \hil_1\to \kil\to \hil_2\to 0
$$
where $\kil$ is another Hilbert module and each map is a module
morphism. Rather than formally defining the equivalence relation and
the group operation, we simply use the following characterization
from \cite{carlson1995}.
\begin{theorem}\label{t-extchar}
Let $(\hil_1, T_1)$ and $(\hil_2, T_2)$ be Hilbert modules. Then,
the group
$$
\Ext_{A(\D)}^1(T_2,T_1)
$$
is isomorphic to $\A/\mathcal{J}$, where $\A$ is the space of
operators $X:\hil_2\to \hil_1$ for which the operator
$$
\left(
\begin{array}{cc}
T_1 & X \\
0 & T_2
\end{array}
\right)
$$
is polynomially bounded, and $\mathcal{J}$ is the space of operators
of the form $T_1L-LT_2$ for some bounded operator $L:\hil_2\to
\hil_1$.
\end{theorem}
If the operator $X:\hil_2\to\hil_1$ belongs to the space $\A$, we denote by $[X]$ its equivalence class in 
$$
\A/\mathcal{J}=\Ext_{A(\D)}^1(T_2,T_1).
$$
It is well-known that given $[X]\in \Ext_{A(\D)}^1(T_2,T_1)$, then $[X]=0$ if and only if the operator 
$$
\left(
\begin{array}{cc}
T_1 & X \\
0 & T_2
\end{array}
\right)
$$
is similar to $T_1\oplus T_2$. Moreover, extension groups are invariant under similarity, so if $(\hil_1',T_1')$
and $(\hil_2',T_2')$ are Hilbert modules which are similar to
$(\hil_1,T_1)$ and $(\hil_2,T_2)$ respectively, then the groups
$\Ext_{A(\D)}^1(T_2,T_1)$ and $\Ext_{A(\D)}^1(T_2',T_1')$ are
isomorphic. A Hilbert module $(\hil_2,T_2)$ is
said to be \emph{projective} if
$$
\Ext_{A(\D)}^1(T_2,T_1)=0
$$
for every Hilbert module $(\hil_1,T_1)$ . It is easy to verify using Theorem
\ref{t-extchar} that the map $[X]\mapsto [X^*]$ establishes an
isomorphism between the groups $\Ext^1_{A(\D)}(T_2,T_1)$ and
$\Ext^1_{A(\D)}(T_1^*,T_2^*)$, so $T_2$ is projective if and only if
$$
\Ext_{A(\D)}^1(T_1,T_2^*)=0
$$
for every Hilbert module $(\hil_1,T_1)$.
 
Because of their connection with commutant lifting properties, those Hilbert modules which are projective are of special interest from the point of view of operator theory. In fact, an important question is whether or not the projectivity of a module can be detected from its basic operator theoretic properties. This problem attracted a
lot of interest  (see \cite{carlson1995}, \cite{carlson1997}, \cite{carlson1994}, \cite{chen2000}, \cite{clancy1998}, \cite{clouatreMODULE1}, \cite{didas2006}, \cite{ferguson1996}, \cite{ferguson1997}, \cite{guo1999} for partial results), but to this day the complete picture is still unclear and the full answer unknown.

Most of the main results about projective modules over the disc algebra focus on the case where said modules are assumed to be similar to a contraction. The only known instance of a projective Hilbert module is when the underlying operator is (similar to) a unitary (see \cite{carlson1994}). On the other hand, Ferguson showed in \cite{ferguson1997} that any module which is projective and similar to a contraction must in fact be similar to an isometry.
Of course, this does not tell the whole story as it is known that there exist polynomially bounded operators that are not similar to a contraction (see \cite{pisier1997}) and thus Hilbert modules that are not similar to a contractive module. 

Our aim is to exhibit necessary conditions for a general Hilbert module $(\hil,T)$ to be projective.  Our main results in this direction say that for such a module, the left spectrum $\sigma_l(T)$
must be contained in the unit circle. This fact can be recovered from Ferguson's result for contractive Hilbert modules, but again the point here is that we do not assume that the module $T$ is similar to a contraction. Furthermore, we obtain those restrictions on the spectrum of the operator $T$ under a variety of assumptions which are formally weaker than projectivity.  
More precisely, we prove the following in Section 2.
\begin{theorem}\label{t-projspecintro}
Let $\lambda\in \D$, let $(\hil,T)$ be a Hilbert module and let $P$ be the orthogonal projection of $\hil$ onto $\ker(T-\lambda)$. If
$\ker(T-\lambda)\neq 0$, then
$$
\Ext^1_{A(\D)}(T,(I-P)T^*(I-P))\neq 0.
$$
\end{theorem}
\begin{theorem}\label{t-lspecintro}
Let $\lambda\in \D$  and let $(\hil,T)$ be a Hilbert module. Then,
$$
\Ext^1_{A(\D)}(T,\lambda)=0
$$
if and only if $\lambda$ does not belong to $\sigma_l(T)$.
\end{theorem}

A notion related to the study of extension groups is that of a derivation of the disc algebra. Recall that given a Hilbert module $(\hil,T)$, a bounded linear map
$$
\delta:A(\D)\to B(\hil)
$$ is called
a \textit{derivation} if it satisfies
$$
\delta(fg)=f(T)\delta(g)+\delta(f)g(T)
$$
for every $f,g\in A(\D)$. A derivation is \textit{inner} if
there exists $\Delta\in B(\hil)$ such that
$$
\delta(f)=f(T)\Delta-\Delta f(T)
$$
for every $f\in A(\D)$. The connection between derivations and extension groups is realized as follows. Let $X\in B(\hil)$ and set 
$$
R=\left(
\begin{array}{cc}
T & X\\
0 & T
\end{array}
\right).
$$
For every polynomial $p$, we have that
$$
p(R)=\left(
\begin{array}{cc}
p(T) & \delta_X(p)\\
0 & p(T)
\end{array}
\right)
$$
for some operator $\delta_X(p)$.
Then, the operator $R$ is polynomially bounded
if and only if the map $$p\mapsto\delta_X(p)$$ extends to a derivation on $A(\D)$. Moreover, $[X]=0$ in $\Ext_{A(\D)}^1(T,T)$ if and only if $\delta_X$ is an inner derivation.

It is an interesting and non-trivial issue to determine the modules $T$ for which every derivation is inner, or equivalently for which the group $\Ext_{A(\D)}^1(T,T)$ is trivial. In relation to this problem, in Section 3 we investigate the condition
$$
\Ext_{A(\D)}^1(T,T)=0
$$
and its consequences on the operator $T$, and establish the following theorem. 
\begin{theorem}\label{t-structintro}
Let $(\hil,T)$ be a Hilbert module such that $\hil$ is infinite dimensional and
$$
\Ext_{A(\D)}^1(T,T)=0.
$$
Then, the subspaces $\ker T$ and $\ker T^*$ are orthogonal,  and  the subspaces 
$$
\{h\in \hil\ominus \ker T: Th\in \hil\ominus \ker T\}
$$
and
$$
\{h\in \hil\ominus \ker T^*: T^*h\in \hil\ominus \ker T^*\}
$$
are infinite dimensional.
\end{theorem}
A natural strengthening of this result would read as follows: if $\Ext_{A(\D)}^1(T,T)=0$, then $T$ has no eigenvalues inside the unit disc. We verify this in the special cases of normal operators in Section 3 (Theorem \ref{t-normal}), and of matrices and $C_0$ contractions in Section 4 (Lemma \ref{l-matrices} and Theorem \ref{t-C0} respectively).

\emph{Acknowledgements} The author is grateful to the referee for his careful reading of the paper which helped improve the exposition.

\section{Spectral properties and the vanishing of extension groups}
Let $(\hil,T)$ be a Hilbert module. Since the powers of $T$ are uniformly bounded, it is a trivial consequence of the spectral radius formula that $\sigma(T)\subset \ol{\D}$. The aim of
this section is to investigate the relation between the spectrum of $T$ and the vanishing of the group $\Ext_{A(\D)}^1(T,X)$ where $X$ is some fixed module.

Recall now that the left (respectively right) spectrum of an element $a$ in a unital Banach algebra is the set of complex numbers $\lambda$ with the property that $a-\lambda$ is not left (respectively right) invertible. These sets are denoted by $\sigma_l(a)$ and $\sigma_r(a)$ respectively.
If we are dealing with a bounded operator $T$ on some Banach space, then it is well-known that $\sigma_l(T)$ coincides with the set of complex numbers
$\lambda$ with the property that $T-\lambda I$ is not bounded below, while $\sigma_r(T)$ coincides with the set of complex numbers $\lambda$ with the property that $T-\lambda I$ is not surjective.

We first reformulate a result of \cite{davis1974} which yields a sufficient spectral condition for the vanishing of an extension group.

\begin{theorem}\label{t-lrspec}
Let $(\hil_1,T_1)$ and $(\hil_2,T_2)$ be Hilbert modules. Then, 
$$
\Ext^1_{A(\D)}(T_2,T_1)=0
$$
if the sets $\sigma_l(T_2)$ and $\sigma_r(T_1)$ are disjoint.
\end{theorem}
\begin{proof}
It follows at once from Theorem 5 of \cite{davis1974} that the map
$$
B(\hil_2,\hil_1)\to B(\hil_2,\hil_1)
$$
$$
L\mapsto T_1L-LT_2
$$
is surjective under our assumption. The conclusion is then an immediate consequence of Theorem \ref{t-extchar}.
\end{proof}

Before giving an easy consequence of Theorem \ref{t-lrspec}, we need some notation. Let $\E$ be a separable Hilbert space and let $H^2(\E)$ be the Hardy space of (weakly) holomorphic $\E$-valued functions on the unit disc with square summable Taylor coefficients at the origin. Let $S=S_{\E}$ the unilateral shift on $H^2(\E)$ which acts by multiplication by the variable.

\begin{corollary}\label{c-projspecshift}
Let $(\hil,T)$ be a Hilbert module such that $\sigma(T)\subset \D$. Then, $ \Ext^1_{A(\D)}(S_{\E},T)=0. $
\end{corollary}
\begin{proof}
This follows immediately from Theorem \ref{t-lrspec} and from the classical fact that the left spectrum of the unitaleral shift is the unit circle $\T$.
\end{proof}
This result contrasts nicely with a result of Carlson and Clark (Corollary 3.4.2 of \cite{carlson1994}) which says that if $\sigma(T)\subset \D$, then 
the group $\Ext_{A(\D)}^1(T,S)$ is isomorphic to $\hil$, where $S$ denotes the shift of multiplicity one.

The remainder of the section is devoted to finding conditions on the spectrum of a module that are necessary for the vanishing of certain extension groups. We first need an auxiliary result which will simplify some proofs. For $\lambda\in \D$, we set
$$
\phi_{\lambda}(z)=\frac{z-\lambda}{1-\ol{\lambda}z}.
$$
If $(\hil,T)$ is a Hilbert module, then the operator $\phi_{\lambda}(T)$ is bounded since $\sigma(T)\subset \ol{\D}$ as was observed at the beginning of the section.

\begin{lemma}\label{l-mobius}
Let $(\hil_1,T_1)$ and $(\hil_2,T_2)$ be Hilbert modules such that
$$
\Ext_{A(\D)}^1(T_2,T_1)=0.
$$
Then
$$
\Ext_{A(\D)}^1(\phi_{\lambda}(T_2),\phi_{\lambda}(T_1))=0
$$
for every $\lambda\in \D$.
\end{lemma}
\begin{proof}
Assume that the operator
$$
R=\left( \begin{array}{cc}
\phi_{\lambda}(T_1) & X\\
0 & \phi_{\lambda}(T_2)
\end{array}\right)
$$
is polynomially bounded, so that there exists a constant $C>0$ such that 
$$
\|f(R)\|\leq C\|f\|_{\infty}
$$
for every $f\in A(\D)$. Using the fact that $\phi_{-\lambda}\circ \phi_{\lambda}(z)=z$ for every $z\in \D$, we find
$$
\phi_{-\lambda}(R)=\left( \begin{array}{cc}
T_1 & Y\\
0 & T_2
\end{array}\right)
$$
for some operator $Y$. But since $\phi_{-\lambda}$ is an automorphism of the unit disc, we see that $\phi_{-\lambda}(R)$ is also polynomially bounded. Indeed, if $f\in A(\D)$ then we have
\begin{align*}
\|f(\phi_{-\lambda}(R))\|&=\|(f\circ \phi_{-\lambda})(R)\|\\
&\leq C\|f\circ \phi_{-\lambda} \|_{\infty}\\
&= C \|f\|_{\infty}.
\end{align*}
Now, $\Ext_{A(\D)}^1(T_2,T_1)$ is assumed to be trivial, so there exists an invertible operator $W$
with the property that
$$
W\phi_{-\lambda}(R)W^{-1}=T_1\oplus T_2
$$
whence
$$
WRW^{-1}=\phi_{\lambda}(T_1)\oplus \phi_{\lambda}(T_2)
$$
and the element $[X]$ is trivial in
$\Ext_{A(\D)}^1(\phi_{\lambda}(T_2),\phi_{\lambda}(T_1))$.
\end{proof}

Another preliminary lemma is required. Its proof can be found in  \cite{clouatreMODULE1}.

\begin{lemma}\label{l-finitebounded}
Let $(\hil_1, T_1)$ and $(\hil_2, T_2)$ be Hilbert modules. Let
$X:\hil_2 \to \hil_1$ be a bounded operator such that
$T_1^NXT_2^N=0$ for some integer $N\geq 0$. Then, the operator
$R:\hil_1\oplus \hil_2\to \hil_1\oplus \hil_2$ defined as
$$
R=\left(
\begin{array}{cc}
T_1 & X\\
0 & T_2
\end{array}
\right)
$$
is polynomially bounded.
\end{lemma}

We now come to the first main result of this section.

\begin{theorem}\label{t-projspec}
Let $\lambda\in \D$, let $(\hil,T)$ be a Hilbert module and let $P$ be the orthogonal projection of $\hil$ onto $\ker(T-\lambda)$. If
$\ker(T-\lambda)\neq 0$, then
$$
\Ext^1_{A(\D)}(T,(I-P)T^*(I-P))\neq 0.
$$
\end{theorem}
\begin{proof}
Let 
$$
T_{\lambda}=(I-P)\phi_{\ol{\lambda}}(T)(I-P).
$$
Since $\ker \phi_{\lambda}(T)$ is clearly invariant for $T$ we have that
\begin{equation}\label{e-Tlambda}
T_{\lambda}=\phi_{\ol{\lambda}}((I-P)T(I-P)).
\end{equation}
Moreover,
$$
f(T_{\lambda})=(I-P)(f\circ\phi_{\ol{\lambda}})(T)(I-P)
$$
for every $f\in A(\D)$ and thus $T_{\lambda}$ is polynomially bounded.
The operator
$$
R=\left(
\begin{array}{cc}
\phi_{\ol{\lambda}}(T^*)& P\\
0& T_\lambda
\end{array}
\right)
$$
acting on $\hil\oplus \hil$ is also seen to be polynomially
bounded in view of Lemma \ref{l-finitebounded} and of the fact that $P T_{\lambda}=0$. We now proceed to show that
$[P]$ gives rise to a non-trivial element of
$$
\Ext^1_{A(\D)}(T_{\lambda},\phi_{\ol{\lambda}}(T^*)).
$$
Assume on the contrary that there exists $L\in B(\hil)$ such that
$$
P=\phi_{\ol{\lambda}}(T^*)L-LT_{\lambda}.
$$
Note that 
$
T_{\lambda}P=0
$
and
\begin{align*}
P\phi_{\ol{\lambda}}(T^*)&=((\phi_{\ol{\lambda}}(T^*))^*P )^*\\
&=(\phi_{\lambda}(T)P
)^*\\
&=0,
\end{align*}
hence
\begin{align*}
P&=P^3\\
&=P(\phi_{\ol{\lambda}}(T^*)L-LT_{\lambda})P\\
&=0
\end{align*}
which is equivalent to
$
\ker (T-\lambda)
$
being trivial, contrary to assumption. Therefore,
$$
\Ext^1_{A(\D)}(T_{\lambda},\phi_{\ol{\lambda}}(T^*))\neq 0.
$$
Note now that equation (\ref{e-Tlambda}) implies that
$$
\Ext^1_{A(\D)}(\phi_{\ol{\lambda}}((I-P)T(I-P)),\phi_{\ol{\lambda}}(T^*))=
\Ext^1_{A(\D)}(T_{\lambda},\phi_{\ol{\lambda}}(T^*))\neq 0.
$$
Lemma \ref{l-mobius} therefore guarantees that
$$
\Ext^1_{A(\D)}((I-P)T(I-P),T^*)\neq 0
$$
which is equivalent to
$$
\Ext^1_{A(\D)}(T,(I-P)T^*(I-P))\neq 0
$$
and the proof is complete.
\end{proof}
Notice that this theorem offers a simple necessary condition for a Hilbert module $(\hil,T)$ to be projective, namely that the point spectrum $\sigma_p(T)$ (the set of eigenvalues of $T$)
be contained in the unit circle $\T$. 
The following is the second main result of this section.

\begin{theorem}\label{t-lspec}
Let $\lambda\in \D$  and let $(\hil,T)$ be a Hilbert module. Then,
$$
\Ext^1_{A(\D)}(T,\lambda)=0
$$
if and only if $\lambda$ does not belong to $\sigma_l(T)$.
\end{theorem}
\begin{proof}
Assume that
$$
\Ext^1_{A(\D)}(T,\lambda)=0.
$$
The operator
$$
R=\left(
\begin{array}{cc}
0  & I\\
0 & \phi_{\lambda}(T)
\end{array}
\right)
$$
acting on $\hil\oplus \hil$ is easily seen to be polynomially
bounded by virtue of Lemma \ref{l-finitebounded}. Now, Lemma \ref{l-mobius} implies that
$$
\Ext^1_{A(\D)}(\phi_{\lambda}(T),0)=0
$$
and thus we can find $L\in B(\hil)$ such that
$$
I=-L\phi_{\lambda}(T).
$$
Consequently
$$
T-\lambda=\phi_{\lambda}(T)(1-\ol{\lambda}T)
$$
is left invertible and $\lambda\notin \sigma_l(T)$. The converse statement follows immediately from Theorem \ref{t-lrspec}.
\end{proof}

This theorem shows in particular that in order for a Hilbert
module $(\hil,T)$ to be projective, it must satisfy
$$
\sigma_p(T)\subset \sigma_l(T)\subset \T.
$$ Now, the reader might wonder about the relevance of Theorem \ref{t-projspec}
in view of the corresponding statement in Theorem \ref{t-lspec}:
the latter is much simpler to prove and has a more satisfactory
conclusion than the former, while the assumption might not look
stronger. However, the assumption that
$$
\Ext^1_{A(\D)}(T,\lambda)=0
$$
is indeed quite strong, and we proceed to illustrate why. The following proposition will be needed later as well.

\begin{proposition}\label{p-strictcontr}
Let $T_1\in B(\hil_1)$ and $T_2\in B(\hil_2)$ be Hilbert modules and assume that $\|T_1\|<1$. Then,
the operator
$$
R=\left(
\begin{array}{cc}
T_1 & X\\
0 & T_2
\end{array}
 \right)
$$
is polynomially bounded for every bounded operator $X:\hil_2 \to \hil_1$.
\end{proposition}
\begin{proof}
Let $p(z)=\sum_{k=0}^d a_k z^k$. Then, a quick calculation shows that
$$
p(R)=\left(
\begin{array}{cc}
p(T_1) & \delta_X(p)\\
0 & p(T_2)
\end{array}
 \right)
$$
where 
$$
\delta_X(p)=\sum_{k=1}^d a_k \sum_{j=0}^{k-1}T_1^{j}XT_2^{k-1-j}.
$$
Since $T_1$ and $T_2$ are polynomially bounded by assumption, to establish that $R$ is also polynomially bounded we need to show
that  that there exists a constant $C>0$ independent of $p$ such that $\left\|\delta_X(p)\right\|\leq
C\|p\|_{\infty}.$
We see that
\begin{align*}
\delta_X(p)&=\sum_{k=1}^d a_k \sum_{j=0}^{k-1}T_1^{j}XT_2^{k-1-j}\\
&=\sum_{j=0}^{d-1}T_1^j X\left(\sum_{k=j+1}^d a_k T_2^{k-1-j}\right)\\
&=\sum_{j=0}^{d-1} T_1^j X \Pi_{j}(T_2)  
\end{align*}
where
$$
\Pi_j(z)=\sum_{k=j+1}^d a_k z^{k-1-j}
$$
for every $0\leq j \leq d-1$. 

We denote by $D:A(\D)\to A(\D)$ the difference quotient operator defined as
$$
Df(z)=\frac{f(z)-f(0)}{z}
$$
for every $f\in A(\D)$. It is well-known that there exists a constant $M>0$ such that
$$
\|D^n\|\leq M(1+\log n)
$$
for every $n\geq 1$, but we sketch the argument for the convenience of the reader.
Given $f\in A(\D)$, one verifies inductively that
$$
D^n f=\frac{1}{z^n}\left( f(z)-\sum_{k=0}^{n-1}\frac{f^{(k)}(0)}{k!}z^k\right)
$$
for every $n\geq 1$ whence
$$
\|D^n f\|_{\infty}=\left\|f(z)-\sum_{k=0}^{n-1}\frac{f^{(k)}(0)}{k!}z^k
 \right\|_{\infty}.
$$
On the other hand, for every $\theta\in \R$ we see that
\begin{align*}
\sum_{k=0}^{n-1}\frac{f^{(k)}(0)}{k!}e^{ik\theta}
&=\frac{1}{2\pi}\int_{0}^{2\pi}f(e^{it})\left(\sum_{k=-(n-1)}^{n-1} e^{-ikt}e^{ik\theta}\right) dt\\
&=\frac{1}{2\pi}\int_{0}^{2\pi}f(e^{it})\Dir_{n-1}(\theta-t) dt
\end{align*}
where
$$
\Dir_n(t)=\sum_{k=-n}^n e^{ikt}
$$
is the Dirichlet kernel.
Therefore,
$$
\|D^n f\|_{\infty}\leq (1+\|\Dir_{n-1}\|_{1})\|f\|_\infty 
$$
for every $n\geq 1$.
It is a classical fact that $\|\Dir_n\|_{1}$ is comparable to $\log n$ as $n\to \infty$, so there exists a constant $M>0$ such that
$$
\|D^n\|\leq M(1+\log n) 
$$
for every $n\geq 1$.

Back to the problem at hand, we know that there exists a constant $C_2>0$ such that
$$
\|f(T_2)\|\leq C_2 \|f\|_{\infty}
$$
for every $f\in A(\D)$.
Using that $\Pi_j=D^{j+1} p$ for every $0\leq j\leq d-1$, we find
\begin{align*}
\|\delta_X(p)\|&\leq \sum_{j=0}^{d-1} \|T_1\|^j \|X\| \|\Pi_j(T_2)\| \\
&\leq C_2 \sum_{j=0}^{d-1}\|T_1\|^j\|X\|\|D^{j+1} p\|_{\infty}  \\
&\leq \left(\sum_{j=0}^{d-1}(1+\log (j+1)) \|T_1\|^j\right) C_2 M\|X\| \|p\|_{\infty}\\
&\leq \left(\sum_{j=0}^{\infty}(1+\log (j+1)) \|T_1\|^j\right)C_2 M\|X\| \|p\|_{\infty}
\end{align*}
and we are done since the series $$\sum_{j=0}^{\infty}(1+\log (j+1)) \|T_1\|^j$$ is convergent by assumption.
\end{proof}
We wish to mention that the general philosophy behind the calculations above can be extracted from the proof of Lemma 2.3 from \cite{petrovic1992}.

Going back to the discussion started before the proposition, let $(\hil,T)$ be a Hilbert module and $\lambda\in \D$.
If we write
$$
\Ext^1_{A(\D)}(T,\lambda)=\A/\mathcal{J}
$$
as in Theorem \ref{t-extchar}, then we
see that $\A$ is very large. Indeed, it is as large as possible since by Proposition \ref{p-strictcontr} it coincides with $B(\hil)$. Thus, the vanishing of the quotient $\A/\mathcal{J}$ is a rather strong condition. Moreover, the corresponding space $\A$ for $$\Ext^1_{A(\D)}(T,(I-P)T^*(I-P))$$ (see Theorem \ref{t-projspec}) is not as large a priori and thus the vanishing of that extension group appears to be a weaker condition. We feel this provides some intuition as to why the assumption of Theorem \ref{t-lspec} may indeed be stronger than that of Theorem \ref{t-projspec}, and that it explains in part the difference in strength of their conclusions.

\section{The derivation problem: a structure theorem}
The rest of the paper is devoted to the study of Hilbert
modules $(\hil,T)$ for which $\Ext_{A(\D)}^1(T,T)=0$. As was mentioned in the introduction, this is directly related with the derivation problem, and in fact this is one of the motivations for our investigation. First, we prove a structure theorem for such Hilbert modules. We focus here on the case where $\hil$ is infinite dimensional. The easier finite dimensional case is fully solved later on in Lemma \ref{l-matrices}.

\begin{theorem}\label{t-struct}
Let $(\hil,T)$ be a Hilbert module such that $\hil$ is infinite dimensional and
$$
\Ext_{A(\D)}^1(T,T)=0.
$$
Then, the subspaces $\ker T$ and $\ker T^*$ are orthogonal,  and  the subspaces 
$$
\{h\in \hil\ominus \ker T: Th\in \hil\ominus \ker T\}
$$
and
$$
\{h\in \hil\ominus \ker T^*: T^*h\in \hil\ominus \ker T^*\}
$$
are infinite dimensional.
%
%
%
%
\end{theorem}
\begin{proof}
Throughout the proof, we may assume without loss of generality that both $\ker T$ and $\ker T^*$ are non-trivial.  We write $\hil=\ker T\oplus (\hil\ominus \ker T)$ and with respect to this decomposition of the space we have
$$
T=\left(
\begin{array}{cc}
0& X\\
0& Y
\end{array}
\right).
$$
Let
$$
P=P_{\ker T}=\left(
\begin{array}{cc}
I& 0\\
0& 0
\end{array}
\right)
$$
be the orthogonal projection of $\hil$ onto $\ker T$
and consider the operator
$$
R=\left(
\begin{array}{cc}
T& P\\
0& T
\end{array}
\right)
$$
which acts on $\hil\oplus \hil$. Using Lemma \ref{l-finitebounded} and the fact that $TP=0$, we see that $R$ is polynomially bounded. By assumption, there exists $L\in
B(\hil)$ such that $P=TL-LT$. If we write
$$
L=\left(
\begin{array}{cc}
L_{11}& L_{12}\\
L_{21}& L_{22}
\end{array}
\right),
$$
then we find
\begin{equation}\label{eqcommutator}
\left(
\begin{array}{cc}
I& 0\\
0& 0
\end{array}
\right)=\left(
\begin{array}{cc}
XL_{21}& XL_{22}-L_{11}X-L_{12}Y\\
YL_{21}& YL_{22}-L_{22}Y-L_{21}X
\end{array}
\right).
\end{equation}
In particular, there must exist a bounded linear operator 
$$
L_{21}:\ker T\to \hil\ominus \ker T
$$
satisfying
\begin{equation}\label{eqXA}
XL_{21}=I
\end{equation}
and
\begin{equation}\label{eqYA}
YL_{21}=0.
\end{equation}
A consequence of (\ref{eqXA}) is that $X$
is surjective, or $X^*$
is bounded below. 
Taking adjoints in (\ref{eqXA}) and (\ref{eqYA}) we find that
\begin{equation}\label{eqA*X*}
L_{21}^*X^*=I
\end{equation}
and
\begin{equation}\label{eqA*Y*}
L_{21}^*Y^*=0.
\end{equation}
Choose $h\in \ol{X^* \ker T}\cap \ol{Y^* (\hil\ominus \ker T)}$. Then
$$
h=\lim_{n\to \infty} X^* v_n=\lim_{n\to \infty} Y^* w_n
$$
for some sequences $\{v_n\}_n\subset \ker T$ and $\{w_n\}_n\subset \hil\ominus \ker T$. Using (\ref{eqA*X*}) and (\ref{eqA*Y*}) we get
$$
L_{21}^*h=\lim_{n\to \infty}v_n=0
$$
so that $h=0$. This shows that
$$
\ol{X^* \ker T}\cap \ol{Y^* (\hil\ominus \ker T)}=\{0\}.
$$
Now, we have that a vector $h=h_1\oplus h_2\in \ker T\oplus (\hil\ominus \ker T)$ lies in $\ker T^*$ if and only if
$$
X^* h_1=-Y^*h_2 \in X^* \ker T\cap Y^* (\hil\ominus \ker T).
$$
Since this intersection was already found to be zero, we see that $h_1\in \ker X^*$ and $h_2\in \ker Y^*$. But $X^*$ is bounded below, whence $h_1=0$ and therefore
$$
\ker T^*=0\oplus \ker Y^*\subset \hil\ominus \ker T
$$
which establishes the first statement.
%
%
We now turn to the proof of the second statement. Notice  that in view of (\ref{eqXA}) we have
that the operator
$$
L_{21}X:\hil\ominus \ker T \to \hil\ominus \ker T
$$
is a non-zero idempotent which we denote henceforth by $E$. With respect to the decomposition
$$
\hil\ominus \ker T=\ran E\oplus (\ran E)^\perp
$$
we can write
$$
E= \left(
\begin{array}{cc}
I& F\\
0& 0
\end{array}
\right)
$$
where $\ran E$ denotes the range of $E$.
If we consider the invertible operator
$$
W= \left(
\begin{array}{cc}
I& F\\
0& I
\end{array}
\right)
$$
then we have
$$
WEW^{-1}= \left(
\begin{array}{cc}
I& 0\\
0& 0
\end{array}
\right).
$$
Now, using (\ref{eqcommutator}) we see
that 
$$
YL_{22}-L_{22}Y=E.
$$
Since $\hil$ is infinite dimensional, by a classical theorem of Wintner (see \cite{wintner1947}) we know that $E$ cannot
be written as the sum of a non-zero scalar multiple of the identity
and a compact operator. The same is necessarily true for
$WEW^{-1}$, whence the orthogonal projection onto
$$
\ker (WEW^{-1})=W\ker E=W\ker X
$$
cannot be compact. In other words,
$$
\ker X=\{h\in \hil\ominus \ker T: Th\in \hil\ominus \ker T\}
$$
is infinite dimensional. We can apply the same argument
to $T^*$ to conclude that
$$
\{h\in \hil\ominus \ker T^*: T^*h\in \hil\ominus \ker T^*\}
$$
is also infinite dimensional, which finishes the proof.
\end{proof}

We make a few comments about this result. By Lemma \ref{l-mobius}, we may replace $T$ by $\phi_{\lambda}(T)$ everywhere in the statement of Theorem \ref{t-struct} and thus obtain information about $\ker(T-\lambda)$ and $\ker (T^*-\ol{\lambda})$ for each $\lambda\in \D$. Interestingly, the theorem provides evidence that the spaces $\ker (T-\lambda)$ and $\ker (T^*-\ol{\lambda})$ cannot be too large under the condition
$\Ext_{A(\D)}^1(T,T)=0$. While we don't know at the moment whether or not these spaces must be trivial in general, the following conjecture seems natural: if $\Ext_{A(\D)}^1(T,T)=0$, then $T$ has no eigenvalues inside the unit disc. 

Next, we consider a special class of operators and prove a weaker version of this conjecture for them. We restrict our attention to the so-called $D$-\emph{symmetric} operators which were introduced and studied in \cite{anderson1978},\cite{rosentrater1981} and \cite{stampfli1979}. Recall that an operator $T\in B(\hil)$ is said to be $D$-symmetric if
$$
\ol{\{TL-LT:L\in B(\hil)\}}=\ol{\{T^*L-LT^*:L\in B(\hil\}}.
$$
It was proved in \cite{anderson1978} that the class of $D$-symmetric operators includes normal operators and isometries.

\begin{theorem}\label{thmDsymm}
Let $(\hil,T)$ be a Hilbert module satisfying
$\Ext_{A(\D)}^1(T,T)=0$. If $T$ is $D$-symmetric, then one of the spaces $\ker T$ and $\ker T^*$ is trivial.
\end{theorem}
\begin{proof}
Assume that we can find unit vectors $f\in \ker T$ and $g\in \ker T^*$, and define $V\in B(\hil)$ as $Vx=\langle
x,g \rangle f$ for every $x\in \hil$. Consider the operator
$$
R= \left( \begin{array}{cc}
T & V\\
0 & T 
\end{array}\right)
$$
which is polynomially bounded by virtue of Lemma \ref{l-finitebounded} since $TV=0$. Notice now that for every $L\in B(\hil)$ we have
$$
\langle (T^*L-LT^*)g,f \rangle=0
$$
by choice of $f$ and $g$, while $\langle Vg,f \rangle=1$. Thus, $V$ lies outside the set
$$
\ol{\{T^*L-LT^*:L\in B(\hil)\}}.
$$
Since $T$ is assumed to be $D$-symmetric, this set coincides with
$$
\ol{\{TL-LT:L\in B(\hil)\}}
$$
and therefore $V$ cannot be written as $TL-LT$ for some
$L\in B(\hil)$, whence $[V]$ is a non-trivial element in $\Ext_{A(\D)}^1(T,T)$.
\end{proof}

Note that the trick used in the proof above to construct the operator $V$ lying outside
the set
$$\ol{\{T^*L-LT^*:L\in B(\hil)\}}$$ is due to Stampfli and can be found in \cite{stampfli1973}. We close this section by specializing even further and verifying the full conjecture for normal operators.

\begin{theorem}\label{t-normal}
Let $(\hil,T)$ be a Hilbert module such that $T$ is normal. Then,
$T$ is unitary if and only if
$$
\Ext_{A(\D)}^1(T,T)=0.
$$
\end{theorem}
\begin{proof}
If $T$ is unitary then the module $(\hil,T)$  is projective by Theorem 4.1 of \cite{carlson1994}, so in particular we have
$$
\Ext_{A(\D)}^1(T,T)=0.
$$
Conversely, assume that this extension group vanishes. If $\lambda\in
\sigma(T)\cap \D$, then via the spectral theorem for normal operators we can find a non-zero
reducing subspace $M\subset \hil$ for $T$ such that
$
\|T|M\|<1.
$
With respect to the decomposition $\hil=M\oplus M^{\perp}$, we have
$T=T|M\oplus T|M^{\perp}$. Consider the operator $X=I\oplus 0$. It is easy to verify that the operator
$$
R=\left(
\begin{array}{cc}
T& X\\ 0& T
\end{array}
 \right)
$$
is unitarily equivalent to
$$
\left(
\begin{array}{cc}
T|M& I\\ 0& T|M
\end{array}
 \right)\oplus
 \left(
\begin{array}{cc}
T|M^\perp& 0\\ 0& T|M^\perp
\end{array}
 \right).
$$
Using Proposition
\ref{p-strictcontr}, we see that $R$ is polynomially bounded and thus
$$
[X]\in \Ext_{A(\D)}^1(T,T).
$$
Since we assume that this extension group is zero, we can write
$$
I\oplus 0=X=TL-LT
$$
for some $L\in B(\hil)$. A straightforward calculation shows that
this relation implies
$$
I=(T|M)L'-L'(T|M)
$$
for some operator $L':M\to M$, which is impossible since the identity is
well-known not to be a commutator (see \cite{wintner1947}). This contradiction shows that
$\sigma(T)\subset \T$, and thus the normal operator $T$ is actually
unitary.
\end{proof}

\section{Contractions of class $C_0$}
In this final section, we verify the conjecture made in Section 3 for another special class of operators: the $C_0$
contractions. We start with some background (see \cite{bercovici1988} or \cite{nagy2010} for greater detail).

Let $H^\infty$ be the algebra of bounded holomorphic functions on
the open unit disc. A completely non-unitary contraction $T\in B(\hil)$ is said to be \textit{of class} $C_0$ if the associated Sz.-Nagy--Foias
$H^\infty$ functional calculus has non-trivial kernel. It is known in that case that 
$$
\{u\in H^\infty: u(T)=0\}=\theta H^\infty
$$ 
for some inner function $\theta$ called the minimal
function of $T$ which is uniquely determined up to
a scalar factor of absolute value one. Moreover, we have that 
$$
\sigma_p(T)=\sigma(T)\cap \D
$$
and this set coincides with the set of zeros of $\theta$ on $\D$.

For any inner function $\theta\in H^\infty$, the space
$H(\theta)=H^2\ominus \theta H^2$ is closed and invariant for $S^*$,
the adjoint of the shift operator $S$ on $H^2$. The operator
$S(\theta)$ defined by $S(\theta)^*=S^*|(H^2\ominus \theta H^2)$ is
called a Jordan block; it is of class $C_0$ with minimal
function $\theta$. We record a well-known elementary property of these
operators.

\begin{lemma}\label{l-jbsim}
Let $\theta_1,\theta_2\in H^\infty$ be inner functions such that
$\theta_1 H^\infty +\theta_2 H^\infty =H^\infty$. Then,
$S(\theta_1\theta_2)$ is similar to $S(\theta_1)\oplus S(\theta_2)$.
\end{lemma}

A more general family of operators consists of the so-called
Jordan operators. Start with a collection of inner functions $\Theta=\{\theta_{\alpha}\}_{\alpha}$ indexed by the ordinal numbers such that $\theta_\alpha=1$ for $\alpha$ large enough and that $\theta_{\beta}$ divides $\theta_{\alpha}$ whenever $\text{card}(\beta)\geq \text{card}(\alpha)$ (recall that a
function $u\in H^\infty$ divides another function $v\in H^\infty$ if
$v=uf$ for some $f\in H^\infty$). Let $\gamma$ be the first ordinal such that $\theta_\gamma=1$. Then, the associated Jordan operator is $J_\Theta=\bigoplus_{\alpha<\gamma} S(\theta_{\alpha})$.

The Jordan
operators are of fundamental importance in the study of operators of
class $C_0$ as the following theorem from \cite{bercovici1975}
illustrates. Recall here that an injective bounded linear operator with dense range is called a quasiaffinity. Two operators $T\in B(\hil)$ and $T'\in B(\hil')$ are said to be
quasisimilar if there exist quasiaffinities $X:\hil\to
\hil'$ and $Y:\hil' \to \hil$ such that $XT=T' X$ and $T Y=YT'$.

\begin{theorem}\label{t-classification}
For any operator $T$ of class $C_0$ there exists a unique Jordan operator
$J$ which is quasisimilar to $T$.
\end{theorem}

With these preliminaries out of the way, we return to the problem at hand. We start with the simple case where the space $\hil$ is finite dimensional, thus complementing Theorem \ref{t-struct}.

\begin{lemma}\label{l-matrices}
Let $(\C^n, T)$ be a Hilbert module. Then, $$\Ext_{A(\D)}^1(T,T)=0$$
if and only if $T$ is similar to a unitary.
\end{lemma}
\begin{proof}
As before, if $T$ is similar to a unitary then by Theorem 4.1 of \cite{carlson1994} we know that the module $(\C^n, T)$ is projective and thus
$$\Ext_{A(\D)}^1(T,T)=0.$$
Assume conversely that this extension group vanishes. This condition is
invariant under similarity, so we may assume in addition that $T$ is of the form
$$
T=J_{\lambda_1,m_1}\oplus\ldots\oplus J_{\lambda_d,m_d}
$$
where $J_{\lambda,m}$ is the usual $m\times m$ Jordan cell with
eigenvalue $\lambda$. Suppose that one of the eigenvalues lies inside
$\D$. In other words, we have $ T=J\oplus T' $ where
$J=J_{\lambda,m}$ for some $\lambda\in \D$ and $1\leq m \leq n$.
Correspondingly, define $X=I\oplus 0$.
It is easy to verify that the operator
$$
R=\left(
\begin{array}{cc}
T& X\\ 0& T
\end{array}
 \right)
$$
is unitarily equivalent to
$$
\left(
\begin{array}{cc}
J& I\\ 0& J
\end{array}
 \right)\oplus
 \left(
\begin{array}{cc}
T'& 0\\ 0& T'
\end{array}
 \right).
$$
Applying a polynomial $p$ to the operator
$$
\left(\begin{array}{cc}
J& I\\ 0& J
\end{array}
 \right)
$$
yields
$$
\left(\begin{array}{cc}
p(J)& p'(J)\\ 0& p(J)
\end{array}
 \right).
$$
On the other hand, an easy computation shows that
$$
f(J)=\left(
\begin{array}{ccccc}
f(\lambda) & f'(\lambda) & f''(\lambda)/2 &\ldots & f^{(m-1)}(\lambda)/(m-1)!\\
0 & f(\lambda) & f'(\lambda) & \ldots & f^{(m-2)}(\lambda)/(m-2)!\\
0 &  0 & f(\lambda) & \ldots & f^{(m-3)}(\lambda)/(m-3)!\\
\vdots & \vdots &  & \ddots & \vdots \\ 
0 & 0 & 0 & \cdots& f(\lambda)
\end{array}
\right)
$$
for every $f\in A(\D)$.
Since $|\lambda|<1$, the classical Cauchy estimates for derivatives of holomorphic functions imply that the operator
$$
\left(\begin{array}{cc}
J& I\\ 0& J
\end{array}
 \right)
$$
is polynomially bounded, and thus so is $R$.

Now, $X$ has non-zero trace and thus cannot be
written as $TL-LT$ for some $L\in B(\C^n)$. Equivalently, $X$ gives
rise to a non-trivial element of $\Ext_{A(\D)}^1(T,T)$, which is a
contradiction. Thus, $\sigma(T)\subset \T$. Since a Jordan cell
$J_{\lambda,m}$ is power-bounded only when $|\lambda|<1$ or $m=1$,
we conclude that every Jordan cell of $T$ has size one, whence $T$ is diagonalizable and hence similar to a
unitary.
\end{proof}

We now tackle the general case where $T\in B(\hil)$ is of class $C_0$.
We begin with an elementary fact.

\begin{lemma}\label{lemmasubsp}
Let $M_1,M_2\subset \hil$ be two closed subspaces with trivial intersection such that $M_1$
has finite dimension. Then, the operator $R:
M_1\oplus M_2\to M_1+ M_2$ defined as $R(m_1\oplus m_2)=m_1+m_2$
is bounded and invertible.
\end{lemma}
\begin{proof}
It is clear $R$ is surjective, and it is injective as well since $M_1\cap M_2=\{0\}$.  A straightforward estimate shows that $R$ is bounded. Since $M_1$ is finite dimensional and $M_2$ is closed, the algebraic sum $M_1+M_2$ is closed and thus $R$ is invertible.
\end{proof}

We need one more preliminary tool. The result is well-known but we provide a proof for the reader's convenience.

\begin{lemma}\label{lemmaC0sim}
Let $T\in B(\hil)$ be an operator of class $C_0$ such that
$\lambda\in \sigma(T)\cap \D$. Then, $T$ is similar
$J_{\lambda,n}\oplus T'$ for some $n\geq 1$ and some operator $T'$.
\end{lemma}
\begin{proof}
If we denote by $\theta$ the minimal function of $T$ and we set as before
$$
\phi_{\lambda}(z)=\frac{z-\lambda}{1-\ol{\lambda}z}
$$
then we can write $\theta=\phi^n_{\lambda} \psi$ where $\psi(\lambda)\neq 0$. It is clear that
$$
\inf_{z\in \D}\{|\psi(z)|+|\phi_{\lambda}^n(z)|\}>0
$$
so by Carleson's corona theorem (see \cite{carleson1958}) we conclude that
$$
\phi^n_{\lambda} H^\infty+\psi H^\infty=H^\infty.
$$
By virtue of Lemma \ref{l-jbsim}, we have that $S(\theta)$ is similar to $S(\phi^n_{\lambda})\oplus S(\psi)$. Now, if $J$ denotes the Jordan model of $T$, then this discussion shows that $J$ is similar to $S(\phi^n_{\lambda}) \oplus J'$ for some operator $J'$, and by Theorem \ref{t-classification} we have that $T$ is quasisimilar to $S(\phi^n_{\lambda})\oplus J'$. If we denote the space on which $S(\phi^n_{\lambda})\oplus J'$ acts by $\kil=H(\phi_{\lambda}^n)\oplus \kil'$, then we can find a quasiaffinity 
$$
Y: \kil\to \hil
$$
such that 
$$
Y(S(\phi^n_{\lambda})\oplus J')=TY.
$$
Let 
$$
M_1=\ol{Y(H(\phi^n_{\lambda})\oplus 0)}
$$
and
$$
M_2=\ol{Y(0\oplus \kil')}.
$$
By Lemma \ref{lemmasubsp}, we have that the operator $R: M_1\oplus M_2\to M_1+ M_2$ defined as $R(m_1\oplus m_2)=m_1+m_2$ is bounded and invertible, and it obviously intertwines $T$ with $T|M_1\oplus T|M_2$. Hence, $T$ is similar to $T|M_1\oplus T|M_2$. But $M_1$ is finite dimensional and the minimal polynomial of $T|M_1$ is clearly $(z-\lambda)^n$, so we find that $T$ is similar to $J_{\lambda,n}\oplus T'$.
\end{proof}

Finally, we come to the main result of this section. Although weaker, it is reminiscent of both Lemma \ref{l-matrices} and Theorem \ref{t-normal}.

\begin{theorem}\label{t-C0}
Let $T\in B(\hil)$ be an operator of class $C_0$ such that
$$\Ext_{A(\D)}^1(T,T)=0.$$ Then, the spectrum of $T$ lies on the unit
circle.
\end{theorem}
\begin{proof}
Assume that $\lambda\in \sigma(T)\cap \D$. The condition
$\Ext_{A(\D)}^1(T,T)=0$ is invariant under similarity, so by Lemma
\ref{lemmaC0sim} we may assume that $T$ is of the form
$T=J_{\lambda,n}\oplus T'$ for some $n\geq 1$ and some operator
$T'$. By Lemma \ref{l-matrices}, we have that
$$
\Ext_{A(\D)}^1(J_{\lambda,n},J_{\lambda,n})\neq 0
$$
so that there exists an operator $X$ with the property that
$$
\left(\begin{array}{cc}
J_{\lambda,n}& X\\ 0& J_{\lambda,n}
\end{array}
 \right)
$$
is polynomially bounded but
$$
X\neq J_{\lambda,n}L-LJ_{\lambda,n}
$$
for every $L$. Consider now $Y=X\oplus 0$ and
$$
R=\left(
\begin{array}{cc}
T& Y\\ 0& T
\end{array}
 \right).
$$
The operator $R$ is unitarily equivalent to
$$
\left(
\begin{array}{cc}
J_{\lambda,n}& X\\ 0& J_{\lambda,n}
\end{array}
 \right)\oplus
 \left(
\begin{array}{cc}
T'& 0\\ 0& T'
\end{array}
 \right)
$$
and thus it is polynomially bounded. Suppose now that there exists an operator $A$ such that 
$$
Y=TA-AT.
$$
A straightforward calculation shows that this relation implies
$$
X=J_{\lambda,n}A'-A'J_{\lambda,n}
$$
for some operator $A'$,
which is absurd. Hence, $[Y]$ yields a non-trivial element of $\Ext_{A(\D)}^1(T,T).$
\end{proof}

In conclusion, we remark that the main results obtained in this paper extend what was already known about the spectrum of \emph{contractive }projective modules. Indeed, we mentioned in the introduction that every such module is (similar to) an isometry, and isometries do not have point spectrum in the unit disc. This is exactly the type of behavior described in Theorems \ref{t-projspec}, \ref{t-lspec}, \ref{t-struct}, \ref{t-normal} and \ref{t-C0}. Moreover, we reiterate that our results were obtained for modules which are not necessarily similar to a contractive one, and under conditions that are formally weaker than projectivity.

\bibliography{/home/raphael/Dropbox/Research/biblio}
\bibliographystyle{plain}

\end{document}